\documentclass[oneside,english]{amsart}
\usepackage[T1]{fontenc}
\usepackage[latin9]{inputenc}
\usepackage{color}
\usepackage{mathrsfs}
\usepackage{amsthm}
\usepackage{amssymb}

\makeatletter
\numberwithin{equation}{section}
\numberwithin{figure}{section}
\theoremstyle{plain}
\newtheorem{thm}{\protect\theoremname}
\theoremstyle{plain}
\newtheorem{lem}[thm]{\protect\lemmaname}
\theoremstyle{remark}
\newtheorem{rem}[thm]{\protect\remarkname}

\RequirePackage{fix-cm}

\makeatother

\usepackage{babel}
\providecommand{\lemmaname}{Lemma}
\providecommand{\remarkname}{Remark}
\providecommand{\theoremname}{Theorem}

\begin{document}
\title[The fundamental Lepage form in two independent variables]{The fundamental Lepage form in two independent variables: a generalization
using order-reducibility}
\author{Zbyn\v{e}k Urban\and Jana Voln\'a}
\address{Z. Urban and J. Voln\'a\\
Department of Mathematics, Faculty of Civil Engineering\\
V\v{S}B-Technical University of Ostrava\\
Ludv\'{i}ka Pod\'{e}\v{s}t\v{e} 1875/17, 708 33 Ostrava-Poruba,
Czech Republic}
\email{zbynek.urban@vsb.cz; jana.volna@vsb.cz}
\begin{abstract}
A second-order generalization of the fundamental Lepage form of geometric
calculus of variations over fibered manifolds with $2$-dimensional
base is described by means of insisting on (i) equivalence relation
``\emph{Lepage differential $2$-form is closed if and only if the
associated Lagrangian is trivial}'' and, (ii) the principal component
of Lepage form, extending the well-known Poincar\'e\textendash Cartan
form, preserves order prescribed by a~given~Lagrangian. This approach
completes several attempts of finding a Lepage equivalent of a~second-order
Lagrangian possessing condition (i), which is well-known for first-order
Lagrangians in field theory due to Krupka and Betounes.
\end{abstract}

\maketitle

\section{Introduction}

Lepage forms play a~basic role in the calculus of variations
of both simple- and multiple-integral problems over fibered manifolds
and Grassmann fibrations. Among the well-known examples of Lepage
forms we mention namely: the Cartan form of classical mechanics
and its generalization in higher-order mechanics (Krupka \cite{Folia});
in first-order field theory, the Poincar\'{e}\textendash Cartan form
(Garc\'{i}a \cite{Garcia}), the Carath\'{e}odory form (Carath\'{e}odory
\cite{Caratheodory}), and the fundamental Lepage form (also known
as the Krupka\textendash Betounes form) (Krupka \cite{Krupka-FundLepEq},
Betounes \cite{Betounes}); in second-order field theory, the generalized
Poincar\'{e}\textendash Cartan form (Krupka \cite{Krupka-Lepage}),
the generalized Carath\'{e}odory form (Crampin and Saunders \cite{CramSaun2},
Urban and Voln\'{a} \cite{UVcar}), the fundamental Lepage form for
second-order, homogeneous Lagrangians (Saunders and Crampin \cite{Saunders}).
See also Gotay \cite{Gotay}, Goldschimdt and Sternberg \cite{Goldschmidt},
Rund \cite{Rund}, Dedecker \cite{Dedecker}, Hor\'{a}k and Kol\'{a}\v{r}
\cite{HorakKolar}, Krupka \cite{Krupka-Lepage}, Krupka and \v{S}t\v{e}p\'{a}nkov\'{a}
\cite{KruStep}, Saunders \cite{SaunBook}, Sniatycki \cite{Sniatycki}.
Further recent attempts of generalization and study of the fundamental
Lepage equivalent for first- and second-order Lagrangians include
also \cite{Palese,Javier,UrbBra}. For a review of basic properties
and results, see Krupka, Krupkov\'{a}, and Saunders \cite{KKS}. 

Replacing the initial Lagrangian by its Lepage equivalent, the corresponding
variational functional is preserved, and in addition the basic variational
properties as variations, extremals, or conservation laws can be formulated
and studied using geometric operations (such as the exterior derivative,
the Lie derivative) acting on the corresponding Lepage equivalent
of a Lagrangian.

Our aim in this note is to study a generalization of the fundamental
Lepage equivalent $Z_{\lambda}$ of a~\emph{second-order} Lagrangian
$\lambda$. This Lepage equivalent was introduced for \emph{first-order}
Lagrangians in variational theory over fibered manifolds with an $n$-dimensional
base (see \cite{Krupka-FundLepEq,Betounes}), and it obeys the following
crucial property:
\begin{equation}
dZ_{\lambda}=0\quad\mathrm{if\,and\,only\,if}\quad E_{\lambda}=0,\label{eq:Equivalence}
\end{equation}
that is, the Lepage equivalent of a Lagrangian is closed if and only
if the Lagrangian is trivial (i.e., the corresponding Euler\textendash Lagrange
expressions vanish identically). 

For $2$-dimensional base (i.e., two independent variables), we show
that a~Lepage equivalent $Z_\lambda$ of a~second-order Lagrangian $\lambda$, obeying the
aforementioned equivalence property, does \emph{not} exist in general when $Z_\lambda=\Theta_\lambda$ \emph{plus} a~2-contact part, where $\Theta_\lambda$ is the principal component (Lepage form) of $Z_\lambda$. Nevertheless, we describe here the fundamental Lepage form, associated with a~second-order Lagrangian which assures that the principal component
of a~Lepage form (the generalized Poincar\'{e}\textendash Cartan
form) has the \emph{same order} as the initial Lagrangian. This order
reducibility assumption is also motivated by the first-order theory
and include, for instance, important class of Lagrangians \emph{linear}
in second derivatives. 

Recent studies on the fundamental Lepage form include namely Saunders
and Crampin \cite{Saunders} (for two independent variables, generalization
of the fundamental form is given for higher-order, homogeneous Largangians
on tangent bundles), and Palese, Rossi, and Zanello \cite{Palese}
(on the basis of integration by parts, possible generalization of
the fundamental form for a second-order Lagrangian is discussed which,
however, {\em differs} from our result for $n=2$). 

Basic underlying geometric structures, well adapted to the present
paper, can be found in book chapters Voln\'a and Urban \cite{Volna},
and Krupka \cite{Handbook}. Throughout, we use the standard geometric
concepts: the exterior derivative $d$, the contraction $i_{\xi}\rho$
of a differential form $\rho$ with respect to a vector field $\xi$,
and the pull-back operation $*$ acting on differential forms.

If $(U,\varphi)$, $\varphi=(x^{i})$, is a chart on smooth manifold
$X$, the local volume element is denoted by ${\color{black}\omega_{0}}=dx^{1}\wedge\ldots\wedge dx^{n}$,
and we put 
\[
{\color{black}\omega_{j}}{\color{black}=i_{\partial/\partial x^{j}}\omega_{0}=\frac{1}{(n-1)!}\varepsilon_{ji_{2}\ldots i_{n}}dx^{i_{2}}\wedge\ldots\wedge dx^{i_{n}},}
\]
where $\varepsilon_{i_{1}i_{2}\ldots i_{n}}$ is the Levi-Civita permutation
symbol. We denote by $Y$ a~fibered manifold of dimension $n+m$
over an $n$-dimensional base manifold $X$ with projection $\pi:Y\rightarrow X$
the surjective submersion. $J^{r}Y$ denotes the $r$-th order jet
prolongation of $Y$ whose elements are $r$-jets $J_{x}^{r}\gamma$
of sections $\gamma$ of $\pi$ with source at $x\in X$ and target
at $\gamma(x)\in Y$. The canonical jet bundle projection is denoted
by $\pi^{r,s}:J^{r}Y\rightarrow J^{s}Y$. Every fibered chart $(V,\psi)$,
$\psi=(x^{i},y^{\sigma})$, $1\leq i\leq n$, $1\leq\sigma\leq m$,on
$Y$ induces the associated chart $(U,\varphi),\varphi=(x^{i})$ on
$X$, and the associated fibered chart $(V^{r},\psi^{r})$ on $J^{r}Y$,
where $U=\pi(V)$, $V^{r}=(\pi^{r,0})^{-1}(V)$, and $\psi^{r}=(x^{i},y^{\sigma},y_{j}^{\sigma},\ldots,y_{j_{1}\ldots j_{r}}^{\sigma})$,
where
\[
y_{j_{1}\ldots j_{k}}^{\sigma}(J_{x}^{r}\gamma)=D_{j_{1}}\ldots D_{j_{k}}(y^{\sigma}\gamma\varphi^{-1})(\varphi(x)),\quad0\leq k\leq r.
\]
A~tangent vector $\xi\in T_{y}Y$ is called $\pi$\emph{-vertical},
if $T\pi\cdot\xi=0$, and a~differential form $\rho$ on $Y$ is
called $\pi$-\emph{horizontal}, if for every point $y\in Y$ the
contraction $i_{\xi}\rho(y)$ vanishes whenever $\xi\in T_{y}Y$ is
$\pi$-vertical.

We denote by $\Omega_{q}^{r}Y$ the $\Omega_{0}^{r}Y$-module of smooth
differential $q$-forms defined on $J^{r}Y$. $\pi^{r}$-horizontal
$q$-forms on $J^{r}Y$ constitute a~submodule of $\Omega_{q}^{r}Y$,
denoted by $\Omega_{q,X}^{r}Y$. For a~fibered manifold $\pi:Y\rightarrow X$
there exists a~unique morphism $h:\Omega^{r}Y\rightarrow\Omega^{r+1}Y$
of exterior algebras of differential forms such that for any fibered
chart $(V,\psi)$, $\psi=(x^{i},y^{\sigma})$, on $Y$, and any differentiable
function $f:J^{r}Y\rightarrow\mathbb{R}$, 
\[
hf=f\circ\pi^{r+1,r},\quad\quad hdf=(d_{i}f)dx^{i},
\]
where $d_{i}$ (resp. $d_{i}^{\prime}$) is the $i$-th formal derivative
(resp. the cut $i$-th formal derivative) operator associated with
$(V,\psi)$,
\begin{equation*}
d_{i}=d_{i}^{\prime}+\sum_{j_{1}\leq\ldots\leq j_{r}}\frac{\partial}{\partial y_{j_{1}\ldots j_{r}}^{\sigma}}y_{j_{1}\ldots j_{r}i}^{\sigma},\label{eq:FormalDerivative}
\end{equation*}
and
\begin{equation}
d_{i}^{\prime}=\frac{\partial}{\partial x^{i}}+\sum_{k=0}^{r-1}\sum_{j_{1}\leq\ldots\leq j_{k}}\frac{\partial}{\partial y_{j_{1}\ldots j_{k}}^{\sigma}}y_{j_{1}\ldots j_{k}i}^{\sigma}.\label{eq:Cut}
\end{equation}

A~differential form $q$-form $\rho\in\Omega_{q}^{r}Y$ satisfying
$h\rho=0$ is called \emph{contact}, and every contact form $\rho$
is generated by contact $1$-forms
\begin{equation*}
\omega_{j_{1}\ldots j_{s}}^{\sigma}=dy_{j_{1}\ldots j_{s}}^{\sigma}-y_{j_{1}\ldots j_{s}i}^{\sigma}dx^{i},\qquad0\leq s\leq r-1.\label{eq:contact}
\end{equation*}
Any differential $q$-form $\rho\in\Omega_{q}^{r}Y$ has a~unique
invariant decomposition,
\[
(\pi^{r+1,r})^{*}\rho=h\rho+\sum_{k=1}^{q}p_{k}\rho,
\]
where $p_{k}\rho$ is the $k$-\emph{contact component} of $\rho$,
containing exactly $k$ exterior product factors $\omega_{j_{1}\ldots j_{s}}^{\sigma}$
with respect to any fibered chart $(V,\psi)$.
 
\section{Lepage equivalents in field theory}

We summarize basic facts about Lepage differential forms on finite-order
jet prolongations of fibered manifolds and, in particular, we discuss
distinguished examples of Lepage equivalents of first- and second-order
Lagrangians; for more details see \cite{Krupka-Lepage,Handbook,Krupka-Book,Volna,UrbBra,UVcar}. 

By a~\textit{Lagrangian} $\lambda$ for a~fibered manifold $\pi:Y\rightarrow X$
of order $r$ we mean an element of the submodule $\Omega_{n,X}^{r}Y$
of $\pi^{r}$-horizontal $n$-forms in the module of $n$-forms $\Omega_{n}^{r}Y$,
defined on the $r$-th jet prolongation $J^{r}Y$. In a~fibered chart
$(V,\psi)$, $\psi=(x^{i},y^{\sigma})$, Lagrangian $\lambda\in\Omega_{n,X}^{r}Y$
has an expression
\begin{equation}
\lambda=\mathscr{L}\omega_{0},\label{eq:Lagrangian}
\end{equation}
where $\omega_{0}=dx^{1}\wedge\ldots\wedge dx^{n}$ is the (local)
volume element, and $\mathscr{L}:V^{r}\rightarrow\mathbb{R}$ is said
to be the \textit{Lagrange function} associated to $\lambda$ and
$(V,\psi)$.

An $n$-form $\rho\in\Omega_{n}^{s}Y$ on $J^{s}Y$ is called a\textit{~Lepage
form,} if one of the following equivalent conditions is satisfied: 

(i) $p_{1}d\rho$ is a $\pi^{s+1,0}$-horizontal $(n+1)$-form, 

(ii) $hi_{\xi}d\rho=0$ for arbitrary $\pi^{s,0}$-vertical vector
field $\xi$ on $J^{s}Y$,

(iii) For every fibered chart $(V,\psi)$ on $Y$, $\rho$ satisfies
\[
(\pi^{s+1,s})^{*}\rho=f_{0}\omega_{0}+\sum_{k=0}^{s}f_{\sigma}^{i,j_{1}\ldots j_{k}}\omega_{j_{1}\ldots j_{k}}^{\sigma}\wedge\omega_{i}+\eta,
\]
where $n$-form $\eta$ has order of contactness $\geq2$, and
\begin{align*}
\frac{\partial f_{0}}{\partial y_{j_{1}\ldots j_{k}}^{\sigma}}-d_{i}f_{\sigma}^{i,j_{1}\ldots j_{k}}-f_{\sigma}^{j_{k},j_{1}\ldots j_{k-1}} & =0\qquad\mathrm{Sym}(j_{1}\ldots j_{k}),\,\,k\leq s,\\
\frac{\partial f_{0}}{\partial y_{j_{1}\ldots j_{s+1}}^{\sigma}}-f_{\sigma}^{j_{s+1},j_{1}\ldots j_{s}} & =0\qquad\mathrm{Sym}(j_{1}\ldots j_{s+1}).
\end{align*}

Let $\lambda\in\Omega_{n,X}^{r}Y$ be a~Lagrangian for $\pi:Y\rightarrow X$.
A~Lepage form $\rho\in\Omega_{n}^{s}Y$ is called a~\emph{Lepage
equivalent }of $\lambda$, if $h\rho=\lambda$ (up to a~canonical
jet projection). The following theorem describes the structure of
Lepage equivalents of a~Lagrangian.
\begin{thm}
\label{Thm:LepageEquiv}Let $\lambda\in\Omega_{n,X}^{r}Y$ be a~Lagrangian
of order $r$ for $\pi:Y\rightarrow X$, locally expressed by \eqref{eq:Lagrangian}
with respect to a~fibered chart $(V,\psi)$. An $n$-form $\rho\in\Omega_{n}^{s}Y$
is a~Lepage equivalent of $\lambda$ if and only if it obeys the
following decomposition,
\begin{equation}
(\pi^{s+1,s})^{*}\rho=\Theta_{\lambda}+d\mu+\eta,\label{eq:LepDecomp}
\end{equation}
where $n$-form $\Theta_{\lambda}$ is defined on $V^{2r-1}$ by
\begin{equation}
\Theta_{\lambda}=\mathscr{L}\omega_{0}+\sum_{k=0}^{r-1}\left(\sum_{l=0}^{r-1-k}(-1)^{l}d_{p_{1}}\ldots d_{p_{l}}\frac{\partial\mathscr{L}}{\partial y_{j_{1}\ldots j_{k}p_{1}\ldots p_{l}i}^{\sigma}}\right)\omega_{j_{1}\ldots j_{k}}^{\sigma}\wedge\omega_{i},\label{eq:PoiCar}
\end{equation}
$\mu$ is a~contact $(n-1)$-form, and an $n$-form $\eta$ has the
order of contactness $\geq2$.
\end{thm}

\begin{proof}
See \cite{Krupka-Book,Volna}.
\end{proof}
$\Theta_{\lambda}$, given by \eqref{eq:PoiCar} on $V^{2r-1}$, is
called the \emph{principal} \emph{Lepage equivalent} of $\lambda$
with respect to fibered chart $(V,\psi)$. This~Lepage form is uniquelly
determined by imposing that a~Lepage form is $\pi^{2r-1,r-1}$-horizontal
and it has the order of contactness \textit{$\leq1$.}\textit{\emph{
We n}}ote that, in general, decomposition \eqref{eq:LepDecomp} is
\emph{not} uniquely determined with respect to contact forms $\mu$
and $\eta$, although the Lepage equivalent $\rho$ satisfying \eqref{eq:LepDecomp}
is a~globally defined differential form on $J^{s}Y$. For a~first-order
Lagrangian $\lambda$, $\Theta_{\lambda}$ \eqref{eq:PoiCar}
is the well-known \emph{Poincar\'e\textendash Cartan form} defined
on $J^{1}Y$ (cf. \cite{Garcia}),
\begin{equation}
\Theta_{\lambda}=\mathscr{L}\omega_{0}+\frac{\partial\mathscr{L}}{\partial y_{j}^{\sigma}}\omega^{\sigma}\wedge\omega_{j}.\label{eq:Poincare-Cartan}
\end{equation}
 For a~second-order Lagrangian $\lambda$, $\Theta_{\lambda}$
\eqref{eq:PoiCar} is the~\emph{generalized Poincar\'e\textendash Cartan
form} defined on $J^{3}Y$ (cf. \cite{Krupka-Lepage,Handbook}),
\begin{equation}
\Theta_{\lambda}=\mathscr{L}\omega_{0}+\left(\frac{\partial\mathscr{L}}{\partial y_{j}^{\sigma}}-d_{i}\frac{\partial\mathscr{L}}{\partial y_{ij}^{\sigma}}\right)\omega^{\sigma}\wedge\omega_{j}+\frac{\partial\mathscr{L}}{\partial y_{ij}^{\sigma}}\omega_{i}^{\sigma}\wedge\omega_{j}.\label{eq:Poincare-Cartan-2ndOrder}
\end{equation}
We point out that for Lagrangians of order $r\geq3$, local expressions
\eqref{eq:PoiCar} need \emph{not} define differential forms on $J^{2r-1}Y$
globally (cf. \cite{HorakKolar,Krupka-Lepage}).

The well-known \emph{Euler\textendash Lagrange mapping} of the calculus
of variations assigns to a~Lagrangian $\lambda\in\Omega_{n,X}^{r}Y$
the \emph{Euler\textendash Lagrange form
\begin{equation}
E_{\lambda}=E_{\sigma}(\mathscr{L})\omega^{\sigma}\wedge\omega_{0},\label{eq:E-L--form}
\end{equation}
}with coefficients
\begin{equation}
E_{\sigma}(\mathscr{L})=\sum_{k=0}^{r}(-1)^{k}d_{i_{1}}\ldots d_{i_{k}}\frac{\partial\mathscr{L}}{\partial y_{i_{1}\ldots i_{k}}^{\sigma}}\label{eq:E-L--expressions}
\end{equation}
the \emph{Euler\textendash Lagrange expressions} associated to $\mathscr{L}:V^{r}\rightarrow\mathbb{R}$.
Note that the $1$-contact and $\pi^{r,0}$-horizontal $(n+1)$-form
$E_{\lambda}$ is defined by means of \eqref{eq:E-L--form}, \eqref{eq:E-L--expressions}
on $J^{2r}Y$ globally. The following theorem explains a~crucial
relation between a~Lepage equivalent of Lagrangian $\lambda$ on
one side and the associated Euler\textendash Lagrange form $E_{\lambda}$
on the other side.
\begin{thm}
\label{Thm:Lepage--ELform}Let $\lambda\in\Omega_{n,X}^{r}Y$ be a~Lagrangian
of order $r$ for $\pi:Y\rightarrow X$ and let $\rho\in\Omega_{n}^{s}Y$
be a~Lepage equivalent of $\lambda$. Then
\begin{equation*}
(\pi^{s+1,s})^{*}d\rho=E_{\lambda}+F,\label{eq:Lep-EL}
\end{equation*}
where $E_{\lambda}$ is the Euler\textendash Lagrange form \eqref{eq:E-L--form}
associated to $\lambda$, and $F$ is an $(n+1)$-form with order
of contactness $\geq2$. In particular , $E_{\lambda}$ coincides
with the $1$-contact component of the exterior derivative of a~Lepage
equivalent of Lagrangian $\lambda$, i.e. on $J^{2r}Y$,
\begin{equation}
E_{\lambda}=p_{1}d\rho.\label{eq:E-PD}
\end{equation}
\end{thm}

\begin{proof}
See \cite{Handbook,Krupka-Book,Volna}.
\end{proof}
Beside the principal Lepage equivalent $\Theta_{\lambda}$, given
by \eqref{eq:Poincare-Cartan} and \eqref{eq:Poincare-Cartan-2ndOrder}
for a~first- and second-order Lagrangian $\lambda$, respectively,
we recall the other known examples of Lepage equivalents, determined
by means of additional requirements.
\begin{lem}
\textbf{\textup{\label{Lem:Caratheodory} }}\emph{(a)} Let $\lambda\in\Omega_{n,X}^{1}Y$
be a~non-vanishing first-order Lagrangian for $\pi:Y\rightarrow X$,
locally expressed by \eqref{eq:Lagrangian}. Then the local expression
\begin{align}
\Lambda_{\lambda} & =\frac{1}{\mathscr{L}^{n-1}}\bigwedge_{j=1}^{n}\left(\mathscr{L}dx^{j}+\frac{\partial\mathscr{L}}{\partial y_{j}^{\sigma}}\omega^{\sigma}\right)\label{eq:CaratheodoryForm}
\end{align}
defines a~$\pi^{1,0}$-horizontal differential $n$-form $\Lambda_{\lambda}\in\Omega_{n}^{1}Y$,
which is a Lepage equivalent of $\lambda$.

\emph{(b)} Let $\lambda\in\Omega_{n,X}^{2}Y$ be a~non-vanishing
second-order Lagrangian for $\pi:Y\rightarrow X$, locally expressed
by \eqref{eq:Lagrangian}. Then the local expression
\begin{align}
\Lambda_{\lambda} & =\frac{1}{\mathscr{L}^{n-1}}\bigwedge_{j=1}^{n}\left(\mathscr{L}dx^{j}+\left(\frac{\partial\mathscr{L}}{\partial y_{j}^{\sigma}}-d_{i}\frac{\partial\mathscr{L}}{\partial y_{ij}^{\sigma}}\right)\omega^{\sigma}+\frac{\partial\mathscr{L}}{\partial y_{ij}^{\sigma}}\omega_{i}^{\sigma}\right)\label{eq:2ndCaratheodoryExpression}
\end{align}
defines a~$\pi^{3,1}$-horizontal differential $n$-form $\Lambda_{\lambda}\in\Omega_{n}^{3}Y$,
which is a Lepage equivalent of $\lambda$. 
\end{lem}

\begin{proof}
See \cite{Caratheodory,UVcar}.
\end{proof}
$\Lambda_{\lambda}$ \eqref{eq:CaratheodoryForm} is the well-known
\emph{Carath\'eodory form, }associated to~a~non-vanishing, first-order
Lagrangian\emph{ $\lambda$ }(cf. \cite{Caratheodory})\emph{,} whereas\emph{
}$\Lambda_{\lambda}$ \eqref{eq:2ndCaratheodoryExpression} is its
generalization for second-order Lagrangians, recently studied in \cite{UVcar}.
\begin{rem}
Note that the Carath\'eodory form $\Lambda_{\lambda}$ \eqref{eq:2ndCaratheodoryExpression}
is decomposed as a~sum of the generalized Poincar\'e\textendash Cartan
form $\Theta_{\lambda}$ \eqref{eq:Poincare-Cartan-2ndOrder} and
a~$\pi^{3,1}$-horizontal, $2$-contact differential $n$-form. For
further purpose, we give this decomposition explicitly for the dimension
of base $n=2$:
\begin{align}
\Lambda_{\lambda} & =\Theta_{\lambda}\nonumber \\
 & +\frac{1}{\mathscr{L}}\left(\frac{\partial\mathscr{L}}{\partial y_{1}^{\sigma}}-d_{i}\frac{\partial\mathscr{L}}{\partial y_{i1}^{\sigma}}\right)\left(\frac{\partial\mathscr{L}}{\partial y_{2}^{\nu}}-d_{k}\frac{\partial\mathscr{L}}{\partial y_{k2}^{\nu}}\right)\omega^{\sigma}\wedge\omega^{\nu}\nonumber \\
 & +\frac{1}{\mathscr{L}}\left(\frac{\partial\mathscr{L}}{\partial y_{j2}^{\nu}}\left(\frac{\partial\mathscr{L}}{\partial y_{1}^{\sigma}}-d_{i}\frac{\partial\mathscr{L}}{\partial y_{i1}^{\sigma}}\right)-\frac{\partial\mathscr{L}}{\partial y_{j1}^{\nu}}\left(\frac{\partial\mathscr{L}}{\partial y_{2}^{\sigma}}-d_{i}\frac{\partial\mathscr{L}}{\partial y_{i2}^{\sigma}}\right)\right)\omega^{\sigma}\wedge\omega_{j}^{\nu}\label{eq:Carat-n2}\\
 & +\frac{1}{\mathscr{L}}\frac{\partial\mathscr{L}}{\partial y_{i1}^{\sigma}}\frac{\partial\mathscr{L}}{\partial y_{j2}^{\nu}}\omega_{i}^{\sigma}\wedge\omega_{j}^{\nu}.\nonumber 
\end{align}
\end{rem}

\begin{lem}
\textbf{\textup{\label{Lem:Fundamental}}} Let $\lambda\in\Omega_{n,X}^{1}Y$
be a~first-order Lagrangian for $\pi:Y\rightarrow X$, locally expressed
by \eqref{eq:Lagrangian}. There exists a~unique Lepage equivalent
$Z_{\lambda}\in\Omega_{n}^{1}Y$ of $\lambda$, which satisfies $Z_{\lambda}=(\pi^{1,0})^{*}\rho$
for any $n$-form $\rho\in\Omega_{n}^{0}W$ on $W$ such that $h\rho=\lambda$.
With respect to a fibered chart $(V,\psi)$, $Z_{\lambda}$ has an
expression 
\begin{align}
Z_{\lambda} & =\mathscr{L}\omega_{0}+\sum_{k=1}^{n}\frac{1}{(n-k)!}\frac{1}{(k!)^{2}}\frac{\partial^{k}\mathscr{L}}{\partial y_{j_{1}}^{\sigma_{1}}\ldots\partial y_{j_{k}}^{\sigma_{k}}}\varepsilon_{j_{1}\ldots j_{k}i_{k+1}\ldots i_{n}}\label{eq:Fundamental}\\
 & \quad\cdot\omega^{\sigma_{1}}\land\ldots\wedge\omega^{\sigma_{k}}\wedge dx^{i_{k+1}}\wedge\ldots\wedge dx^{i_{n}}.\nonumber 
\end{align}
\end{lem}

\begin{proof}
See \cite{Krupka-FundLepEq,Betounes}.
\end{proof}
$Z_{\lambda}$ \eqref{eq:Fundamental} is known as the \emph{fundamental
Lepage form,} or\emph{ Krupka\textendash Betounes form,} associated
to first-order Lagrangian \emph{$\lambda$;} original sources are
\cite{Krupka-FundLepEq} and \cite{Betounes}, further recent contributions
include \cite{Palese,Javier,UrbBra}.
\begin{rem}
One can directly verify that the Lepage equivalent \textit{$Z_{\lambda}$}
\eqref{eq:Fundamental} satisfies the equivalence
relation ``\textit{$Z_{\lambda}$}\emph{ is closed if and only if
the associated Lagrangian $\lambda=hZ_{\lambda}$ is trivial}'',
that is ``$dZ_{\lambda}=0$ if and only if\textit{ $E_{\lambda}=0$.''}
Since $E_{\lambda}=p_{1}dZ_{\lambda}$ \eqref{eq:E-PD}, it is immediate
that $\lambda\in\Omega_{n,X}^{1}Y$ is trivial, provided \textit{$Z_{\lambda}$}\textit{\emph{
is closed. However, the converse implication if a~non-trivial one.}} A~construction of (local, 1-contact) Lepage equivalents of higher-order Lagrangians satisfying this equivalence relation has been recently described in \cite{Voicu}.

Another remarkable property which \textit{$Z_{\lambda}$} \eqref{eq:Fundamental}
satisfies is the following: \textit{$Z_{\lambda}$ is $\pi^{1,0}$-projectable
if and only $E_{\lambda}$ is $\pi^{2,1}$-projectable.}
\end{rem}

\section{Trivial Lagrangians}

A~Lagrangian $\lambda$ is said to be \emph{variationally trivial}
(or \emph{null}), if the associated Euler\textendash Lagrange form
$E_{\lambda}$ vanishes identically. The following theorem describes
variationally trivial Lagrangians.
\begin{thm}
\label{Thm:Trivial} Let $\lambda\in\Omega_{n,X}^{r}Y$ be a~Lagrangian
of order $r$ for $\pi:Y\rightarrow X$. The following conditions
are equivalent:

\emph{(i)} $\lambda$ is trivial.

\emph{(ii)} For every fibered chart $(V,\psi)$ on $Y$ there exists
an $(n-1)$-form $\mu\in\Omega_{n-1}^{r-1}Y$ such that on $V^{r}$,
\[
\lambda=hd\mu.
\]

\emph{(iii)} For every fibered chart $(V,\psi)$ on $Y$ there exist
functions $g^{i}:V^{r}\rightarrow\mathbb{R}$ such that $\lambda=\mathscr{L}\omega_{0}$
on $V^{r}$, where
\begin{equation}
\mathscr{L}=d_{i}g^{i}.\label{eq:TrivialLagrangian}
\end{equation}
\end{thm}

\begin{proof}
See \cite{KrupkaTrivial}; also cf. \cite{Volna}.
\end{proof}
As a~consequence of Theorem \ref{Thm:Trivial}, we now summarize
explicit chart conditions for a~trivial Lagrangian of \emph{second-order},
needed later in this paper. Let $\lambda\in\Omega_{n,X}^{2}Y$ be
a~Lagrangian for $Y$, locally expressed by $\lambda=\mathscr{L}\omega_{0}$
\eqref{eq:Lagrangian} with respect to a~fibered chart $(V,\psi)$,
$\psi=(x^{i},y^{\sigma})$, on $Y$, where $\mathscr{L}=\mathscr{L}(x^{i},y^{\sigma},y_{j}^{\sigma},y_{jk}^{\sigma})$,
$j\leq k$, is the associated Lagrange function defined on $V^{2}\subset J^{2}Y$,
and $E_{\sigma}(\mathscr{L})$ are the corresponding Euler\textendash Lagrange
expressions on $V^{4}\subset J^{4}Y$,
\begin{equation}
E_{\sigma}(\mathscr{L})=\frac{\partial\mathscr{L}}{\partial y^{\sigma}}-d_{i}\frac{\partial\mathscr{L}}{\partial y_{i}^{\sigma}}+d_{i}d_{j}\frac{\partial\mathscr{L}}{\partial y_{ij}^{\sigma}}.\label{eq:ELexpressions2nd}
\end{equation}

\begin{lem}
\label{Thm:Trivial-2} \emph{(a)} $\lambda$ is trivial if and only
if it satisfies the following conditions
\begin{align}
 & \frac{\partial\mathscr{L}}{\partial y^{\sigma}}-d_{i}^{\prime}\frac{\partial\mathscr{L}}{\partial y_{i}^{\sigma}}+d_{i}^{\prime}d_{j}^{\prime}\frac{\partial\mathscr{L}}{\partial y_{ij}^{\sigma}}=0,\label{eq:TC1}\\
 & \frac{\partial^{2}\mathscr{L}}{\partial y_{p}^{\nu}\partial y_{iq}^{\sigma}}-\frac{\partial^{2}\mathscr{L}}{\partial y_{pq}^{\nu}\partial y_{i}^{\sigma}}+2d_{j}^{\prime}\left(\frac{\partial^{2}\mathscr{L}}{\partial y_{pq}^{\nu}\partial y_{ij}^{\sigma}}\right)=0\qquad\mathrm{Sym}(pqi),\label{eq:TC2}\\
 & \frac{\partial^{3}\mathscr{L}}{\partial y_{st}^{\mu}\partial y_{pq}^{\nu}\partial y_{ij}^{\sigma}}=0\qquad\mathrm{Sym}(stj),\,\,\mathrm{Sym}(pqi),\label{eq:TC3}\\
 & \frac{\partial^{2}\mathscr{L}}{\partial y_{pq}^{\nu}\partial y_{ij}^{\sigma}}=0\qquad\mathrm{Sym}(pqij),\label{eq:TC4}
\end{align}
where $d_{i}^{\prime}$ is the cut formal derivative operator \emph{\eqref{eq:Cut}.}

\emph{(b)} For every fibered chart $(V,\psi)$ on $Y$ there exist
functions $g^{i}:V^{2}\rightarrow\mathbb{R}$ such that $\lambda=\mathscr{L}\omega_{0}$
on $V^{2}$, where
\[
\mathscr{L}=d_{i}g^{i},
\]
and
\begin{equation}
\frac{\partial g^{i}}{\partial y_{jk}^{\sigma}}+\frac{\partial g^{j}}{\partial y_{ki}^{\sigma}}+\frac{\partial g^{k}}{\partial y_{ij}^{\sigma}}=0.\label{eq:TCC}
\end{equation}
\end{lem}

\begin{proof}
Consider the Euler\textendash Lagrange expressions $E_{\sigma}(\mathscr{L})$,
associated to a~second-order Lagrangian $\lambda=\mathscr{L}\omega_{0}$.
We have
\begin{align*}
d_{i}\frac{\partial\mathscr{L}}{\partial y_{i}^{\sigma}} & =\frac{\partial^{2}\mathscr{L}}{\partial x^{i}\partial y_{i}^{\sigma}}+\frac{\partial^{2}\mathscr{L}}{\partial y^{\nu}\partial y_{i}^{\sigma}}y_{i}^{\nu}+\frac{\partial^{2}\mathscr{L}}{\partial y_{p}^{\nu}\partial y_{i}^{\sigma}}y_{pi}^{\nu}+\frac{\partial^{2}\mathscr{L}}{\partial y_{pq}^{\nu}\partial y_{i}^{\sigma}}y_{pqi}^{\nu}\\
 & =d_{i}^{\prime}\frac{\partial\mathscr{L}}{\partial y_{i}^{\sigma}}+\frac{\partial^{2}\mathscr{L}}{\partial y_{pq}^{\nu}\partial y_{i}^{\sigma}}y_{pqi}^{\nu},
\end{align*}
and
\begin{align*}
 & d_{i}d_{j}\frac{\partial\mathscr{L}}{\partial y_{ij}^{\sigma}}=d_{i}\left(d_{j}^{\prime}\frac{\partial\mathscr{L}}{\partial y_{ij}^{\sigma}}+\frac{\partial^{2}\mathscr{L}}{\partial y_{pq}^{\nu}\partial y_{ij}^{\sigma}}y_{pqj}^{\nu}\right)\\
 & =d_{i}^{\prime}d_{j}^{\prime}\frac{\partial\mathscr{L}}{\partial y_{ij}^{\sigma}}+d_{j}^{\prime}\left(\frac{\partial^{2}\mathscr{L}}{\partial y_{pq}^{\nu}\partial y_{ij}^{\sigma}}\right)y_{pqi}^{\nu}+\frac{\partial^{2}\mathscr{L}}{\partial y_{p}^{\nu}\partial y_{iq}^{\sigma}}y_{pqi}^{\nu}\\
 & +d_{i}^{\prime}\left(\frac{\partial^{2}\mathscr{L}}{\partial y_{pq}^{\nu}\partial y_{ij}^{\sigma}}\right)y_{pqj}^{\nu}+\frac{\partial^{3}\mathscr{L}}{\partial y_{st}^{\mu}\partial y_{pq}^{\nu}\partial y_{ij}^{\sigma}}y_{sti}^{\mu}y_{pqj}^{\nu}+\frac{\partial^{2}\mathscr{L}}{\partial y_{pq}^{\nu}\partial y_{ij}^{\sigma}}y_{pqji}^{\nu}.
\end{align*}
Thus \eqref{eq:ELexpressions2nd} now reads,
\begin{align}
E_{\sigma}(\mathscr{L}) & =\frac{\partial\mathscr{L}}{\partial y^{\sigma}}-d_{i}^{\prime}\frac{\partial\mathscr{L}}{\partial y_{i}^{\sigma}}+d_{i}^{\prime}d_{j}^{\prime}\frac{\partial\mathscr{L}}{\partial y_{ij}^{\sigma}}\nonumber \\
 & +\left(\frac{\partial^{2}\mathscr{L}}{\partial y_{p}^{\nu}\partial y_{iq}^{\sigma}}-\frac{\partial^{2}\mathscr{L}}{\partial y_{pq}^{\nu}\partial y_{i}^{\sigma}}+2d_{j}^{\prime}\left(\frac{\partial^{2}\mathscr{L}}{\partial y_{pq}^{\nu}\partial y_{ij}^{\sigma}}\right)\right)y_{pqi}^{\nu}\label{eq:ELcoor}\\
 & +\frac{\partial^{3}\mathscr{L}}{\partial y_{st}^{\mu}\partial y_{pq}^{\nu}\partial y_{ij}^{\sigma}}y_{sti}^{\mu}y_{pqj}^{\nu}+\frac{\partial^{2}\mathscr{L}}{\partial y_{pq}^{\nu}\partial y_{ij}^{\sigma}}y_{pqij}^{\nu},\nonumber 
\end{align}
hence $E_{\sigma}(\mathscr{L})$ vanish if and only if conditions
\eqref{eq:TC2} hold, proving (a).

Condition (b) follows from Theorem \ref{Thm:Trivial}, (iii), where
\eqref{eq:TCC} is implied by equation \eqref{eq:TrivialLagrangian}
satisfied identically on $V^{2}\subset J^{2}Y$.
\end{proof}

\section{The fundamental Lepage equivalent of a~second-order Lagrangian: order
reduction}

First, we present an order reduction condition, which allows us to
construct a~generalization of the fundamental Lepage equivalent $Z_{\lambda}$
\eqref{eq:Fundamental} for a~second-order Lagrangian $\lambda$.
From the Theorem \ref{Thm:LepageEquiv} on the structure of every~Lepage
equivalent of a~Lagrangian, it is immediate that a~differential
$n$-form to be found must be decomposable as $Z_{\lambda}=\Theta_{\lambda}+d\mu+\eta$
(up to a~canonical jet projection), where $\Theta_{\lambda}$ is
the generalized Poincar\'e\textendash Cartan form \eqref{eq:Poincare-Cartan-2ndOrder},
$\mu$ is a~contact $(n-1)$-form, and an $n$-form $\eta$ has the
order of contactness $\geq2$. 

Let us assume that $\Theta_{\lambda}$ \eqref{eq:Poincare-Cartan-2ndOrder}
is of \emph{the same order }as the Lagrangian $\lambda$. 

Note that this condition is automatically satisfied for first-order
Lagrangians but, nevertheless, for the second-order it restricts the
class of Lagrangians under consideration.
\begin{lem}
\label{lem:Order}Let $\lambda\in\Omega_{n,X}^{2}Y$ be a~second-order
Lagrangian for $\pi:Y\rightarrow X$. Then the generalized Poincar\'e--Cartan
form $\Theta_{\lambda}$ \eqref{eq:Poincare-Cartan-2ndOrder} is of
second-order, if and only if for every chart $(V,\psi)$, $\psi=(x^{i},y^{\sigma})$,
on $Y$,
\begin{align}
 & \frac{\partial^{2}\mathscr{L}}{\partial y_{pq}^{\nu}\partial y_{ij}^{\sigma}}+\frac{\partial^{2}\mathscr{L}}{\partial y_{ip}^{\nu}\partial y_{qj}^{\sigma}}+\frac{\partial^{2}\mathscr{L}}{\partial y_{qi}^{\nu}\partial y_{pj}^{\sigma}}=0,\label{eq:OrderCondition}
\end{align}
where $\mathscr{L}$ is the Lagrange function, associated to $\lambda$
\eqref{eq:Lagrangian}. 

For $n=2$, \eqref{eq:OrderCondition} read
\begin{align}
 & \frac{\partial^{2}\mathscr{L}}{\partial y_{11}^{\nu}\partial y_{11}^{\sigma}}=0,\,\,\frac{\partial^{2}\mathscr{L}}{\partial y_{11}^{\nu}\partial y_{12}^{\sigma}}=0,\,\,\frac{\partial^{2}\mathscr{L}}{\partial y_{22}^{\nu}\partial y_{21}^{\sigma}}=0,\,\,\frac{\partial^{2}\mathscr{L}}{\partial y_{22}^{\nu}\partial y_{22}^{\sigma}}=0,\nonumber \\
 & \frac{\partial^{2}\mathscr{L}}{\partial y_{11}^{\nu}\partial y_{22}^{\sigma}}+2\frac{\partial^{2}\mathscr{L}}{\partial y_{12}^{\nu}\partial y_{12}^{\sigma}}=0.\label{eq:OrderCond2}
\end{align}
\end{lem}

\begin{proof}
The necessary and sufficient condition \eqref{eq:OrderCondition}
follows immediately from the chart expression \eqref{eq:Poincare-Cartan-2ndOrder}
of $\Theta_{\lambda}$, by means of annihilating terms linear in coordinates
$y_{pqi}^{\tau}$.
\end{proof}
\begin{lem}
\label{lem:Combination}Let $\lambda\in\Omega_{n,X}^{2}Y$ be a~second-order
Lagrangian for $\pi:Y\rightarrow X$ such that the generalized Poincar\'e--Cartan
form $\Theta_{\lambda}$ \eqref{eq:Poincare-Cartan-2ndOrder} is of
second-order. Then $\lambda$ is variationally trivial, if and only
if for every chart $(V,\psi)$, $\psi=(x^{i},y^{\sigma})$, on $Y$,
\begin{align}
 & \frac{\partial\mathscr{L}}{\partial y^{\sigma}}-d_{i}^{\prime}\frac{\partial\mathscr{L}}{\partial y_{i}^{\sigma}}+d_{i}^{\prime}d_{j}^{\prime}\frac{\partial\mathscr{L}}{\partial y_{ij}^{\sigma}}=0,\quad 
 \frac{\partial^{2}\mathscr{L}}{\partial y_{p}^{\nu}\partial y_{iq}^{\sigma}}-\frac{\partial^{2}\mathscr{L}}{\partial y_{pq}^{\nu}\partial y_{i}^{\sigma}}=0\qquad\mathrm{Sym}(pqi).\label{eq:Combination}
\end{align}

For $n=2$, \eqref{eq:Combination} read
\begin{align}
 & \frac{\partial\mathscr{L}}{\partial y^{\sigma}}-d_{i}^{\prime}\frac{\partial\mathscr{L}}{\partial y_{i}^{\sigma}}+d_{i}^{\prime}d_{j}^{\prime}\frac{\partial\mathscr{L}}{\partial y_{ij}^{\sigma}}=0,\quad i,j=1,2,\nonumber \\
 & \frac{\partial^{2}\mathscr{L}}{\partial y_{1}^{\nu}\partial y_{11}^{\sigma}}-\frac{\partial^{2}\mathscr{L}}{\partial y_{11}^{\nu}\partial y_{1}^{\sigma}}=0,\quad\frac{\partial^{2}\mathscr{L}}{\partial y_{2}^{\nu}\partial y_{22}^{\sigma}}-\frac{\partial^{2}\mathscr{L}}{\partial y_{22}^{\nu}\partial y_{2}^{\sigma}}=0,\nonumber \\
 & 2\frac{\partial^{2}\mathscr{L}}{\partial y_{1}^{\nu}\partial y_{12}^{\sigma}}-2\frac{\partial^{2}\mathscr{L}}{\partial y_{1}^{\sigma}\partial y_{12}^{\nu}}+\frac{\partial^{2}\mathscr{L}}{\partial y_{2}^{\nu}\partial y_{11}^{\sigma}}-\frac{\partial^{2}\mathscr{L}}{\partial y_{2}^{\sigma}\partial y_{11}^{\nu}}=0,\label{eq:Combination2}\\
 & 2\frac{\partial^{2}\mathscr{L}}{\partial y_{2}^{\nu}\partial y_{12}^{\sigma}}-2\frac{\partial^{2}\mathscr{L}}{\partial y_{12}^{\nu}\partial y_{2}^{\sigma}}+\frac{\partial^{2}\mathscr{L}}{\partial y_{1}^{\nu}\partial y_{22}^{\sigma}}-\frac{\partial^{2}\mathscr{L}}{\partial y_{22}^{\nu}\partial y_{1}^{\sigma}}=0.\nonumber 
\end{align}
\end{lem}

\begin{proof}
Necessary and sufficient conditions \eqref{eq:Combination} for a~variationally
trivial Lagrangian are nothing but reduction of \eqref{eq:TC1}\textendash \eqref{eq:TC4},
Lemma \ref{Thm:Trivial-2}, with the help of \eqref{eq:OrderCondition}. 
\end{proof}
In the next Theorem we construct the fundamental Lepage equivalent
$Z_{\lambda}$ of a~second-order Lagrangian $\lambda$ over fibered
manifolds $\pi:Y\rightarrow X$, where $\dim X=2$.

Analogously to the Carath\'eodory form $\Lambda_{\lambda}$ \eqref{eq:Carat-n2},
associated to second-order Lagrangian $\lambda\in\Omega_{2,X}^{2}Y$,
suppose that $Z_{\lambda}$ is decomposed as a~sum of the generalized
Poincar\'e\textendash Cartan form $\Theta_{\lambda}$ \eqref{eq:Poincare-Cartan-2ndOrder}
and contact terms, generated by the wedge products $\omega^{\sigma}\wedge\omega^{\nu}$,
$\omega^{\sigma}\wedge\omega_{j}^{\nu}$, and $\omega_{i}^{\sigma}\wedge\omega_{j}^{\nu}$,
i.e.

\begin{align}
 & Z_{\lambda}=\Theta_{\lambda}+\frac{1}{2}P_{\sigma\nu}\omega^{\sigma}\wedge\omega^{\nu}+Q_{\sigma,\nu}^{j}\omega^{\sigma}\wedge\omega_{j}^{\nu}+\frac{1}{2}R_{\sigma,\nu}^{i,j}\omega_{i}^{\sigma}\wedge\omega_{j}^{\nu},\label{eq:FLE}
\end{align}
where $P_{\sigma\nu}$, $Q_{\sigma,\nu}^{j}$, and $R_{\sigma,\nu}^{i,j}$
are real-valued functions on $V^{3}\subset J^{3}Y$ such that $P_{\sigma\nu}$
is skew-symmetric in $\left(\sigma,\nu\right)$, and $R_{\sigma,\nu}^{i,j}$
is skew-symmetric in pairs $\left(i,\sigma\right),\left(j,\nu\right)$.
\begin{thm}
\label{thm:Main}Let $\lambda\in\Omega_{2,X}^{2}Y$ be a~second-order
Lagrangian for $\pi:Y\rightarrow X$ such that \eqref{eq:OrderCondition}
holds. The following two conditions are equivalent:

\emph{(i)} If $\lambda$ is variationally trivial, then $Z_{\lambda}$
\eqref{eq:FLE} is closed.

\emph{(ii) }For every chart $(V,\psi)$, $\psi=(x^{i},y^{\sigma})$,
on $Y$, $Z_{\lambda}$ \eqref{eq:FLE} is uniquelly determined by
means of real-valued functions $P_{\sigma\nu}$, $Q_{\sigma,\nu}^{j}$,
and $R_{\sigma,\nu}^{i,j}$, defined on $V^{2}\subset J^{2}Y$ as
\begin{align}
P_{\sigma\nu} & =\frac{1}{2}\left(\frac{\partial^{2}\mathscr{L}}{\partial y_{1}^{\sigma}\partial y_{2}^{\nu}}-\frac{\partial^{2}\mathscr{L}}{\partial y_{1}^{\nu}\partial y_{2}^{\sigma}}\right)\nonumber \\
 & +d_{1}^{\prime}\left(\frac{\partial^{2}\mathscr{L}}{\partial y_{1}^{\nu}\partial y_{12}^{\sigma}}-\frac{\partial^{2}\mathscr{L}}{\partial y_{1}^{\sigma}\partial y_{12}^{\nu}}\right)+d_{2}^{\prime}\left(\frac{\partial^{2}\mathscr{L}}{\partial y_{2}^{\sigma}\partial y_{12}^{\nu}}-\frac{\partial^{2}\mathscr{L}}{\partial y_{2}^{\nu}\partial y_{12}^{\sigma}}\right),\nonumber \\
Q_{\sigma,\nu}^{1} & =2\frac{\partial^{2}\mathscr{L}}{\partial y_{1}^{\sigma}\partial y_{12}^{\nu}}-\frac{\partial^{2}\mathscr{L}}{\partial y_{1}^{\nu}\partial y_{12}^{\sigma}}-\frac{\partial^{2}\mathscr{L}}{\partial y_{2}^{\nu}\partial y_{11}^{\sigma}}-2d_{2}^{\prime}\frac{\partial^{2}\mathscr{L}}{\partial y_{12}^{\sigma}\partial y_{12}^{\nu}},\label{eq:FLE-Conditions}\\
Q_{\sigma,\nu}^{2} & =-2\frac{\partial^{2}\mathscr{L}}{\partial y_{2}^{\sigma}\partial y_{12}^{\nu}}+\frac{\partial^{2}\mathscr{L}}{\partial y_{1}^{\nu}\partial y_{22}^{\sigma}}+\frac{\partial^{2}\mathscr{L}}{\partial y_{2}^{\nu}\partial y_{12}^{\sigma}}+2d_{1}^{\prime}\frac{\partial^{2}\mathscr{L}}{\partial y_{12}^{\sigma}\partial y_{12}^{\nu}},\nonumber \\
R_{\sigma,\nu}^{1,2} & =-2\frac{\partial^{2}\mathscr{L}}{\partial y_{12}^{\sigma}\partial y_{12}^{\nu}}=-R_{\sigma,\nu}^{2,1},\nonumber \\
R_{\sigma,\nu}^{1,1} & =0,\quad R_{\sigma,\nu}^{2,2}=0,\nonumber 
\end{align}
\end{thm}

\begin{proof}
Suppose that Lagrangian $\lambda\in\Omega_{2,X}^{2}Y$ is variationally
trivial and the generalized Poincar\'e\textendash Cartan form $\Theta_{\lambda}$
\eqref{eq:Poincare-Cartan-2ndOrder} is defined on $J^{2}Y$. Thus,
in abritrary fibered chart $(V,\psi)$, $\psi=(x^{i},y^{\sigma})$,
on $Y$, the associated Lagrange function $\mathscr{L}:V^{2}\rightarrow\mathbb{R}$
satisfies conditions \eqref{eq:TC1}\textendash \eqref{eq:TC3} of
Lemma \ref{Thm:Trivial-2}, and \eqref{eq:OrderCondition} of Lemma
\ref{lem:Order}. Note first that condition \eqref{eq:OrderCondition}
already implies \eqref{eq:TC4}, and using \eqref{eq:OrderCondition},
\[
d_{i}\frac{\partial\mathscr{L}}{\partial y_{ij}^{\sigma}}=d_{i}^{\prime}\frac{\partial\mathscr{L}}{\partial y_{ij}^{\sigma}}.
\]
 Hence, the exterior derivative of $\Theta_{\lambda}$ reads,
\begin{align}
d\Theta_{\lambda} & =d\mathscr{L}\wedge\omega_{0}+d\left(\frac{\partial\mathscr{L}}{\partial y_{j}^{\sigma}}-d_{i}\frac{\partial\mathscr{L}}{\partial y_{ij}^{\sigma}}\right)\wedge\omega^{\sigma}\wedge\omega_{j}+\left(\frac{\partial\mathscr{L}}{\partial y_{j}^{\sigma}}-d_{i}\frac{\partial\mathscr{L}}{\partial y_{ij}^{\sigma}}\right)d\omega^{\sigma}\wedge\omega_{j}\nonumber \\
 & +d\left(\frac{\partial\mathscr{L}}{\partial y_{ij}^{\sigma}}\right)\wedge\omega_{i}^{\sigma}\wedge\omega_{j}+\frac{\partial\mathscr{L}}{\partial y_{ij}^{\sigma}}d\omega_{i}^{\sigma}\wedge\omega_{j}\nonumber \\
 & =\left(\frac{\partial\mathscr{L}}{\partial y^{\sigma}}-d_{j}\frac{\partial\mathscr{L}}{\partial y_{j}^{\sigma}}+d_{i}d_{j}\frac{\partial\mathscr{L}}{\partial y_{ij}^{\sigma}}\right)\omega^{\sigma}\wedge\omega_{0}-\frac{\partial}{\partial y^{\tau}}\left(\frac{\partial\mathscr{L}}{\partial y_{j}^{\sigma}}-d_{i}^{\prime}\frac{\partial\mathscr{L}}{\partial y_{ij}^{\sigma}}\right)\omega^{\sigma}\wedge\omega^{\tau}\wedge\omega_{j}\label{eq:dTheta}\\
 & +\left(\frac{\partial^{2}\mathscr{L}}{\partial y^{\sigma}\partial y_{kj}^{\tau}}-\frac{\partial}{\partial y_{k}^{\tau}}\left(\frac{\partial\mathscr{L}}{\partial y_{j}^{\sigma}}-d_{i}^{\prime}\frac{\partial\mathscr{L}}{\partial y_{ij}^{\sigma}}\right)\right)\omega^{\sigma}\wedge\omega_{k}^{\tau}\wedge\omega_{j}\nonumber \\
 & -\frac{\partial}{\partial y_{kl}^{\tau}}\left(\frac{\partial\mathscr{L}}{\partial y_{j}^{\sigma}}-d_{i}^{\prime}\frac{\partial\mathscr{L}}{\partial y_{ij}^{\sigma}}\right)\omega^{\sigma}\wedge\omega_{kl}^{\tau}\wedge\omega_{j}-\frac{\partial^{2}\mathscr{L}}{\partial y_{k}^{\tau}\partial y_{ij}^{\sigma}}\omega_{i}^{\sigma}\wedge\omega_{k}^{\tau}\wedge\omega_{j}\nonumber \\
 & -\frac{\partial^{2}\mathscr{L}}{\partial y_{kl}^{\tau}\partial y_{ij}^{\sigma}}\omega_{i}^{\sigma}\wedge\omega_{kl}^{\tau}\wedge\omega_{j}.\nonumber 
\end{align}
From \eqref{eq:FLE}, we have
\begin{align*}
dZ_{\lambda} & =d\Theta_{\lambda}\\
 & +\frac{1}{2}dP_{\sigma\nu}\wedge\omega^{\sigma}\wedge\omega^{\nu}+\frac{1}{2}P_{\sigma\nu}d\left(\omega^{\sigma}\wedge\omega^{\nu}\right)+dQ_{\sigma,\nu}^{j}\wedge\omega^{\sigma}\wedge\omega_{j}^{\nu}+Q_{\sigma,\nu}^{j}d\left(\omega^{\sigma}\wedge\omega_{j}^{\nu}\right)\\
 & +\frac{1}{2}dR_{\sigma,\nu}^{i,j}\wedge\omega_{i}^{\sigma}\wedge\omega_{j}^{\nu}+\frac{1}{2}R_{\sigma,\nu}^{i,j}d\left(\omega_{i}^{\sigma}\wedge\omega_{j}^{\nu}\right),
\end{align*}
and combining with \eqref{eq:dTheta} we get a~decomposition of $dZ_{\lambda}$
containing independent base terms, 
\[
dZ_{\lambda}=E_{\lambda}+F_{0}+F_{1}+F_{2},
\]
where $E_{\lambda}$ is the Euler\textendash Lagrange form of $\lambda$,
$F_{0}$ and $F_{1}$ are the $2$-contact parts and $F_{2}$ is the~$3$-contact
part of $dZ_{\lambda}$. For indices $i$, $j$ running through $\{1,2\}$,
and $\kappa(1)=2,$ $\kappa(2)=1$, we have
\begin{align}
F_{0} & =(-1)^{j-1}\left(P_{\sigma\nu}+d_{\kappa(j)}Q_{\sigma,\nu}^{\kappa(j)}+\frac{\partial^{2}\mathscr{L}}{\partial y^{\sigma}\partial y_{12}^{\nu}}-\frac{\partial}{\partial y_{\kappa(j)}^{\nu}}\left(\frac{\partial\mathscr{L}}{\partial y_{j}^{\sigma}}-d_{i}^{\prime}\frac{\partial\mathscr{L}}{\partial y_{ij}^{\sigma}}\right)\right)\nonumber \\
 & \qquad\omega^{\sigma}\wedge\omega_{\kappa(j)}^{\nu}\wedge\omega_{j}\nonumber \\
 & +(-1)^{j-1}\left(Q_{\sigma,\tau}^{\kappa(j)}+\frac{\partial}{\partial y_{\kappa(j)\kappa(j)}^{\tau}}\left(\frac{\partial\mathscr{L}}{\partial y_{j}^{\sigma}}-d_{i}^{\prime}\frac{\partial\mathscr{L}}{\partial y_{ij}^{\sigma}}\right)\right)\omega^{\sigma}\wedge\omega_{\kappa(j)\kappa(j)}^{\tau}\wedge\omega_{j}\nonumber \\
 & +(-1)^{j-1}\left(Q_{\sigma,\tau}^{j}+2(-1)^{j}\frac{\partial}{\partial y_{12}^{\tau}}\left(\frac{\partial\mathscr{L}}{\partial y_{j}^{\sigma}}-d_{i}^{\prime}\frac{\partial\mathscr{L}}{\partial y_{ij}^{\sigma}}\right)\right)\omega^{\sigma}\wedge\omega_{12}^{\tau}\wedge\omega_{j}\label{eq:F0-coef}\\
 & +(-1)^{j-1}\left(R_{\sigma,\nu}^{\kappa(j),j}+(-1)^{j-1}\frac{\partial^{2}\mathscr{L}}{\partial y_{\kappa(j)\kappa(j)}^{\sigma}\partial y_{jj}^{\nu}}\right)\omega_{\kappa(j)\kappa(j)}^{\sigma}\wedge\omega_{j}^{\nu}\wedge\omega_{j}\nonumber \\
 & +(-1)^{j-1}\left(R_{\sigma,\nu}^{j,\kappa(j)}+2(-1)^{j-1}\frac{\partial^{2}\mathscr{L}}{\partial y_{12}^{\sigma}\partial y_{12}^{\nu}}\right)\omega_{12}^{\sigma}\wedge\omega_{\kappa(j)}^{\nu}\wedge\omega_{j}\nonumber \\
 & +(-1)^{j}R_{\sigma,\nu}^{j,j}\omega_{jj}^{\sigma}\wedge\omega_{j}^{\nu}\wedge\omega_{\kappa(j)}+(-1)^{j-1}R_{\sigma,\nu}^{j,j}\omega_{12}^{\sigma}\wedge\omega_{j}^{\nu}\wedge\omega_{j},\nonumber 
\end{align}
and
\begin{align}
F_{1} & =\left(\frac{\partial}{\partial y^{\sigma}}\left(\frac{\partial\mathscr{L}}{\partial y_{j}^{\tau}}-d_{i}^{\prime}\frac{\partial\mathscr{L}}{\partial y_{ij}^{\tau}}\right)+(-1)^{j-1}\frac{1}{2}d_{\kappa(j)}^{\prime}P_{\sigma\tau}\right)\omega^{\sigma}\wedge\omega^{\tau}\wedge\omega_{j}\quad\mathrm{Alt}(\sigma\tau)\nonumber \\
 & +(-1)^{j-1}\left(d_{\kappa(j)}Q_{\sigma,\nu}^{j}+\frac{\partial^{2}\mathscr{L}}{\partial y^{\sigma}\partial y_{jj}^{\nu}}-\frac{\partial}{\partial y_{j}^{\nu}}\left(\frac{\partial\mathscr{L}}{\partial y_{j}^{\sigma}}-d_{i}^{\prime}\frac{\partial\mathscr{L}}{\partial y_{ij}^{\sigma}}\right)\right)\omega^{\sigma}\wedge\omega_{j}^{\nu}\wedge\omega_{j}\nonumber \\
 & -\frac{\partial}{\partial y_{jj}^{\tau}}\left(\frac{\partial\mathscr{L}}{\partial y_{j}^{\sigma}}-d_{i}^{\prime}\frac{\partial\mathscr{L}}{\partial y_{ij}^{\sigma}}\right)\omega^{\sigma}\wedge\omega_{jj}^{\tau}\wedge\omega_{j}\nonumber \\
 & +\frac{1}{2}(-1)^{j-1}\left(Q_{\sigma,\nu}^{\kappa(j)}-Q_{\nu,\sigma}^{\kappa(j)}+(-1)^{j}\left(\frac{\partial^{2}\mathscr{L}}{\partial y_{\kappa(j)}^{\nu}\partial y_{12}^{\sigma}}-\frac{\partial^{2}\mathscr{L}}{\partial y_{\kappa(j)}^{\sigma}\partial y_{12}^{\nu}}\right)\right)\label{eq:F1-Coef}\\
 & \qquad\omega_{\kappa(j)}^{\sigma}\wedge\omega_{\kappa(j)}^{\nu}\wedge\omega_{j}\nonumber \\
 & +\frac{1}{2}(-1)^{j-1}\left(d_{\kappa(j)}R_{\sigma,\nu}^{j,j}+\frac{\partial^{2}\mathscr{L}}{\partial y_{j}^{\sigma}\partial y_{jj}^{\nu}}-\frac{\partial^{2}\mathscr{L}}{\partial y_{j}^{\nu}\partial y_{jj}^{\sigma}}\right)\omega_{j}^{\sigma}\wedge\omega_{j}^{\nu}\wedge\omega_{j}\nonumber \\
 & +(-1)^{j-1}\left(\frac{1}{2}d_{\kappa(j)}\left(R_{\sigma,\nu}^{1,2}-R_{\nu,\sigma}^{2,1}\right)+(-1)^{j}Q_{\nu,\sigma}^{j}+\frac{\partial^{2}\mathscr{L}}{\partial y_{j}^{\sigma}\partial y_{12}^{\nu}}-\frac{\partial^{2}\mathscr{L}}{\partial y_{\kappa(j)}^{\nu}\partial y_{jj}^{\sigma}}\right)\nonumber \\
 & \qquad\omega_{1}^{\sigma}\wedge\omega_{2}^{\nu}\wedge\omega_{j},\nonumber
\end{align}
and
\begin{align}
F_{2}= & \frac{1}{6}\left(\frac{\partial P_{\sigma\nu}}{\partial y^{\tau}}+\frac{\partial P_{\nu\tau}}{\partial y^{\sigma}}+\frac{\partial P_{\tau\sigma}}{\partial y^{\nu}}\right)\omega^{\tau}\wedge\omega^{\sigma}\wedge\omega^{\nu}\nonumber \\
 & +\frac{1}{2}\left(\frac{\partial P_{\sigma\nu}}{\partial y_{j}^{\tau}}-\frac{\partial Q_{\sigma,\tau}^{j}}{\partial y^{\nu}}+\frac{\partial Q_{\nu,\tau}^{j}}{\partial y^{\sigma}}\right)\omega^{\sigma}\wedge\omega^{\nu}\wedge\omega_{j}^{\tau}\nonumber \\
 & +\frac{1}{2}\frac{\partial P_{\sigma\nu}}{\partial y_{jk}^{\tau}}\omega_{jk}^{\tau}\wedge\omega^{\sigma}\wedge\omega^{\nu}+\frac{1}{2}\frac{\partial P_{\sigma\nu}}{\partial y_{jkl}^{\tau}}\omega_{jkl}^{\tau}\wedge\omega^{\sigma}\wedge\omega^{\nu}\label{eq:F2-Coef}
 \end{align}
 \begin{align*}
 & +\frac{1}{2}\left(\frac{\partial R_{\sigma,\nu}^{i,j}}{\partial y^{\tau}}-\frac{\partial Q_{\tau,\nu}^{j}}{\partial y_{i}^{\sigma}}+\frac{\partial Q_{\tau,\sigma}^{i}}{\partial y_{j}^{\nu}}\right)\omega^{\tau}\wedge\omega_{i}^{\sigma}\wedge\omega_{j}^{\nu}\nonumber \\
 & +\frac{\partial Q_{\sigma,\nu}^{j}}{\partial y_{kl}^{\tau}}\omega^{\sigma}\wedge\omega_{j}^{\nu}\wedge\omega_{kl}^{\tau}+\frac{\partial Q_{\sigma,\nu}^{j}}{\partial y_{klp}^{\tau}}\omega^{\sigma}\wedge\omega_{j}^{\nu}\wedge\omega_{klp}^{\tau}\nonumber \\
 & +\frac{1}{2}\frac{\partial R_{\sigma,\nu}^{i,j}}{\partial y_{k}^{\tau}}\omega_{k}^{\tau}\wedge\omega_{i}^{\sigma}\wedge\omega_{j}^{\nu}+\frac{1}{2}\frac{\partial R_{\sigma,\nu}^{i,j}}{\partial y_{kl}^{\tau}}\omega_{kl}^{\tau}\wedge\omega_{i}^{\sigma}\wedge\omega_{j}^{\nu}+\frac{1}{2}\frac{\partial R_{\sigma,\nu}^{i,j}}{\partial y_{klp}^{\tau}}\omega_{klp}^{\tau}\wedge\omega_{i}^{\sigma}\wedge\omega_{j}^{\nu}\nonumber 
\end{align*}
(cf. Theorem \ref{Thm:Lepage--ELform}). As we require that $dZ_{\lambda}$
vanishes, one can now recover concrete formulae of coefficients $P_{\sigma\nu}$,
$Q_{\sigma,\nu}^{j}$, and $R_{\sigma,\nu}^{i,j}$ in the expression
\eqref{eq:FLE} of $Z_{\lambda}$. Indeed, annihilating \eqref{eq:F0-coef}
we obtain the functions \eqref{eq:FLE-Conditions}, which are correctly
defined since $\mathscr{L}$ satisfies \eqref{eq:OrderCond2} and
the symmetric component of 
\[
d_{\kappa(j)}Q_{\sigma,\nu}^{\kappa(j)}+\frac{\partial^{2}\mathscr{L}}{\partial y^{\sigma}\partial y_{12}^{\nu}}-\frac{\partial}{\partial y_{\kappa(j)}^{\nu}}\left(\frac{\partial\mathscr{L}}{\partial y_{j}^{\sigma}}-d_{i}^{\prime}\frac{\partial\mathscr{L}}{\partial y_{ij}^{\sigma}}\right)
\]
is zero. 

It remains to show that all coefficients of forms $F_{1}$ \eqref{eq:F1-Coef} and $F_{2}$ \eqref{eq:F2-Coef}
vanish identically. To this can be, however, directly verified by
means of the assumption on variational triviality of $\lambda$, namely
conditions \eqref{eq:Combination2} of Lemma \ref{lem:Combination},
as well as partial derivative operators $\partial/\partial y_{j}^{\tau}$,
$\partial/\partial y_{jk}^{\tau}$ applied to condition \eqref{eq:TC1}.
In particular, the following identities follow from \eqref{eq:Combination2}
and are employed,
\begin{align*}
 & \frac{\partial^{3}\mathscr{L}}{\partial y_{1}^{\sigma}\partial y_{2}^{\tau}\partial y_{12}^{\nu}}+\frac{\partial^{3}\mathscr{L}}{\partial y_{1}^{\nu}\partial y_{2}^{\sigma}\partial y_{12}^{\tau}}+\frac{\partial^{3}\mathscr{L}}{\partial y_{1}^{\tau}\partial y_{2}^{\nu}\partial y_{12}^{\sigma}}\\
 & \quad-\frac{\partial^{3}\mathscr{L}}{\partial y_{1}^{\tau}\partial y_{2}^{\sigma}\partial y_{12}^{\nu}}-\frac{\partial^{3}\mathscr{L}}{\partial y_{1}^{\sigma}\partial y_{2}^{\nu}\partial y_{12}^{\tau}}-\frac{\partial^{3}\mathscr{L}}{\partial y_{1}^{\nu}\partial y_{2}^{\tau}\partial y_{12}^{\sigma}}=0,\\
 & \frac{\partial^{3}\mathscr{L}}{\partial y_{1}^{\nu}\partial y_{12}^{\sigma}\partial y_{12}^{\tau}}-\frac{\partial^{3}\mathscr{L}}{\partial y_{1}^{\sigma}\partial y_{12}^{\nu}\partial y_{12}^{\tau}}=0,\qquad\frac{\partial^{3}\mathscr{L}}{\partial y_{2}^{\nu}\partial y_{12}^{\sigma}\partial y_{12}^{\tau}}-\frac{\partial^{3}\mathscr{L}}{\partial y_{2}^{\sigma}\partial y_{12}^{\nu}\partial y_{12}^{\tau}}=0,
\end{align*}
\begin{align*}
 & \frac{\partial^{2}\mathscr{L}}{\partial y^{\sigma}\partial y_{1}^{\nu}}-\frac{\partial^{2}\mathscr{L}}{\partial y^{\nu}\partial y_{1}^{\sigma}}+\frac{1}{2}d_{2}^{\prime}\left(\frac{\partial^{2}\mathscr{L}}{\partial y_{1}^{\sigma}\partial y_{2}^{\nu}}-\frac{\partial^{2}\mathscr{L}}{\partial y_{1}^{\nu}\partial y_{2}^{\sigma}}\right)+d_{1}^{\prime}d_{2}^{\prime}\left(\frac{\partial^{2}\mathscr{L}}{\partial y_{1}^{\nu}\partial y_{12}^{\sigma}}-\frac{\partial^{2}\mathscr{L}}{\partial y_{1}^{\sigma}\partial y_{12}^{\nu}}\right)\\
 & +\frac{1}{2}d_{2}^{\prime}d_{2}^{\prime}\left(\frac{\partial^{2}\mathscr{L}}{\partial y_{1}^{\nu}\partial y_{22}^{\sigma}}-\frac{\partial^{2}\mathscr{L}}{\partial y_{1}^{\sigma}\partial y_{22}^{\nu}}\right)+d_{i}^{\prime}\left(\frac{\partial^{2}\mathscr{L}}{\partial y^{\nu}\partial y_{1i}^{\sigma}}-\frac{\partial^{2}\mathscr{L}}{\partial y^{\sigma}\partial y_{1i}^{\nu}}\right)=0,
\end{align*}
\begin{align*}
 & \frac{\partial^{2}\mathscr{L}}{\partial y^{\tau}\partial y_{2}^{\sigma}}-\frac{\partial^{2}\mathscr{L}}{\partial y^{\sigma}\partial y_{2}^{\tau}}+\frac{1}{2}d_{1}^{\prime}\left(\frac{\partial^{2}\mathscr{L}}{\partial y_{2}^{\tau}\partial y_{1}^{\sigma}}-\frac{\partial^{2}\mathscr{L}}{\partial y_{2}^{\sigma}\partial y_{1}^{\tau}}\right)+\frac{1}{2}d_{1}^{\prime}d_{1}^{\prime}\left(\frac{\partial^{2}\mathscr{L}}{\partial y_{2}^{\sigma}\partial y_{11}^{\tau}}-\frac{\partial^{2}\mathscr{L}}{\partial y_{2}^{\tau}\partial y_{11}^{\sigma}}\right)\\
 & +d_{1}^{\prime}d_{2}^{\prime}\left(\frac{\partial^{2}\mathscr{L}}{\partial y_{2}^{\sigma}\partial y_{12}^{\tau}}-\frac{\partial^{2}\mathscr{L}}{\partial y_{2}^{\tau}\partial y_{12}^{\sigma}}\right)+d_{i}^{\prime}\left(\frac{\partial^{2}\mathscr{L}}{\partial y^{\sigma}\partial y_{i2}^{\tau}}-\frac{\partial^{2}\mathscr{L}}{\partial y^{\tau}\partial y_{i2}^{\sigma}}\right)=0,\\
 & \frac{\partial^{2}\mathscr{L}}{\partial y^{\sigma}\partial y_{12}^{\nu}}+\frac{\partial^{2}\mathscr{L}}{\partial y^{\nu}\partial y_{12}^{\sigma}}-\frac{1}{2}\left(\frac{\partial^{2}\mathscr{L}}{\partial y_{1}^{\nu}\partial y_{2}^{\sigma}}+\frac{\partial^{2}\mathscr{L}}{\partial y_{2}^{\nu}\partial y_{1}^{\sigma}}\right)+2d_{1}^{\prime}d_{2}^{\prime}\frac{\partial^{2}\mathscr{L}}{\partial y_{12}^{\nu}\partial y_{12}^{\sigma}}\\
 & -d_{i}^{\prime}\frac{\partial^{2}\mathscr{L}}{\partial y_{i}^{\sigma}\partial y_{12}^{\nu}}+d_{j}^{\prime}\left(\frac{\partial^{2}\mathscr{L}}{\partial y_{2}^{\nu}\partial y_{1j}^{\sigma}}+\frac{\partial^{2}\mathscr{L}}{\partial y_{1}^{\nu}\partial y_{2j}^{\sigma}}\right)=0.
\end{align*}
The equivalence of (i) and (ii) is therefore proved.
\end{proof} 
\begin{rem}
The proof of Theorem \ref{thm:Main} shows that $Z_{\lambda}$ \eqref{eq:FLE}
satisfying the equivalence relation ``$Z_{\lambda}$ is closed if
and only if $\lambda$ is variationally trivial'' does \emph{not}
exist in general. In particular, without the assumption \eqref{eq:OrderCondition}
on order-reducibility of the generalized Poincar\'e--Cartan form $\Theta_{\lambda}$,
the coeffcients $P_{\sigma\nu}$, $Q_{\sigma,\nu}^{j}$, and $R_{\sigma,\nu}^{i,j}$
of $Z_{\lambda}$ are generally \emph{not} well-defined to ensure the closure
condition of $Z_{\lambda}$ for a~trivial Lagrangian $\lambda$. 

As an example illustrating this property, we refer to the {\em Camassa--Holm equation} and its Lagrangian on the second jet prolongation of fibered manifold $S^1\times \mathbb{R} \times \mathbb{R}$ over $S^1\times \mathbb{R}$ (see \cite{Kouranbaeva}), 
\begin{equation}
\mathscr{L}_{CH} = \frac{1}{2} \Big( y_1 (y_2)^2 + \frac{(y_{12})^2}{y_1} \Big),\label{eq:CH}
\end{equation}
which is quadratic in second derivatives and does {\em not} satisfy the order-reducibility condition \eqref{eq:OrderCondition}. Following the proof of Theorem \ref{thm:Main}, it can be directly shown that $Z_{\lambda}$ \eqref{eq:FLE}, associated to Lagrangian \eqref{eq:CH} and satisfying the equivalence relation \eqref{eq:Equivalence}, does {\em not} exist.

Note that the result of Theorem \ref{thm:Main} can be also directly verified
as follows: consider $Z_{\lambda}$ \eqref{eq:FLE} with coefficients
\eqref{eq:FLE-Conditions}, where the associated Lagrange function
$\mathscr{L}$ is of the form $\mathscr{L}=d_{i}g^{i}$, where \eqref{eq:TCC}
of Lemma \ref{Thm:Trivial-2} holds. One can then verify that $dZ_{\lambda}=0$
in a~straightforward way.

A possible generalization of the fundametal form, associated to a second-order Lagrangian in field theory, has been recently studied in \cite{Palese};  the obtained result is, however, not the case described by Theorem 4 since the corresponding Lepage form does not obey the closure condition \eqref{eq:Equivalence}.
\end{rem}

\begin{rem}
Clearly, the order-reducibility condition \eqref{eq:OrderCondition}
imposed beforehand is satisfied for Lagrangians $\lambda\in\Omega_{n,X}^{2}Y$
\emph{linear in second derivatives} $y_{ij}^{\tau}$. In this case,
if 
\[
\mathscr{L}=A+B_{\sigma}^{ij}y_{ij}^{\sigma}=A+B_{\sigma}^{11}y_{11}^{\sigma}+2B_{\sigma}^{12}y_{12}^{\sigma}+B_{\sigma}^{22}y_{22}^{\sigma}
\]
is a~Lagrange function satisfying \eqref{eq:OrderCond2}, then
\begin{align*}
P_{\sigma\nu} & =\frac{1}{2}\left(\frac{\partial^{2}A}{\partial y_{1}^{\sigma}\partial y_{2}^{\nu}}-\frac{\partial^{2}A}{\partial y_{1}^{\nu}\partial y_{2}^{\sigma}}+\left(\frac{\partial^{2}B_{\tau}^{ij}}{\partial y_{1}^{\sigma}\partial y_{2}^{\nu}}-\frac{\partial^{2}B_{\tau}^{ij}}{\partial y_{1}^{\nu}\partial y_{2}^{\sigma}}\right)y_{ij}^{\tau}\right)\\
 & +(-1)^{j-1}2d_{j}^{\prime}\left(\frac{\partial B_{\sigma}^{12}}{\partial y_{j}^{\nu}}-\frac{\partial B_{\nu}^{12}}{\partial y_{j}^{\sigma}}\right),\\
Q_{\sigma,\nu}^{1} & =4\frac{\partial B_{\nu}^{12}}{\partial y_{1}^{\sigma}}-2\frac{\partial B_{\sigma}^{12}}{\partial y_{1}^{\nu}}-\frac{\partial B_{\sigma}^{11}}{\partial y_{2}^{\nu}},\quad Q_{\sigma,\nu}^{2}=-4\frac{\partial B_{\nu}^{12}}{\partial y_{2}^{\sigma}}+2\frac{\partial B_{\sigma}^{12}}{\partial y_{2}^{\nu}}+\frac{\partial B_{\sigma}^{22}}{\partial y_{1}^{\nu}},\\
R_{\sigma,\nu}^{1,1} & =0,\quad R_{\sigma,\nu}^{2,2}=0,\quad R_{\sigma,\nu}^{1,2}=0,\quad R_{\sigma,\nu}^{2,1}=0.
\end{align*}

Moreover, from \eqref{eq:ELcoor} it is easy to see that the order-reducibility \eqref{eq:OrderCondition} is a weaker condition than the requirement on a second-order Lagrangian leading to {\em second-order} Euler--Lagrange equations, cf. \cite{Sardanashvily}, which is the case of the Einstein--Hilbert gravitation Lagrangian of General Relativity. For the example of interaction of gravitational and electromagnetic fields Lagrangian, see  \cite{Volna}.
\end{rem}

Let $\lambda\in\Omega_{2,X}^{2}Y$ be a~Lagrangian for $\pi:Y\rightarrow X$
such that the generalized Poincar\'e-Cartan form $\Theta_{\lambda}$
\eqref{eq:Poincare-Cartan-2ndOrder} is of second-order. We call $Z_{\lambda}$
\eqref{eq:FLE}, satisfying one of the equivalent conditions of Theorem
\ref{thm:Main}, the \emph{fundamental Lepage equivalent}, associated
to a~second-order Lagrangian $\lambda$.

Our conjecture is that this kind of generalization
of the fundamental form using order reducibility can be provided for
a~general dimension of a~base manifold, which will be studied in
a future work.




\vspace{6pt}

\end{document}